\documentclass[12pt, 
english,
]{article}
\usepackage{custom_tex}
\usepackage[most]{tcolorbox}

\usepackage{titlesec}

\titleformat{\section}
  {\large\bfseries}   % font for section
  {\thesection}       % section number format
  {1em}               % spacing between number and title
  {}

\titleformat{\subsection}
  {\normalsize\bfseries} % font for subsection
  {\thesubsection}
  {1em}
  {}

\title{Solving a Research Problem in Mathematical Statistics \\ with AI Assistance}

\author{
Edgar Dobriban\footnote{
Department of Statistics and Data Science, University of Pennsylvania. E-mail address: \texttt{dobriban@wharton.upenn.edu}.}
}
\date{\today}

\begin{document}
\maketitle
\abstract{
Over the last few months, 
AI models including large language models have improved greatly.
There are now several documented examples where they have helped professional mathematical scientists prove new results, sometimes even helping resolve known open problems.  
In this short note, we add another example to the list, by documenting how we were able to solve a previously unsolved research problem in robust mathematical statistics with crucial help from GPT-5. 
Our problem concerns robust density estimation, where the observations are perturbed by Wasserstein-bounded contaminations.
 In a previous preprint (Chao and Dobriban, 2023, arxiv:2308.01853v2), we have obtained upper and lower bounds on the minimax optimal estimation error; which were, however, not sharp. 
 
 Starting in October 2025, making significant use of GPT-5 Pro, we were able to derive the minimax optimal error rate (reported in version 3 of the above arxiv preprint).
 GPT-5 provided crucial help along the way, including  by suggesting calculations that we did not think of,  and techniques that were not familiar to us, such as the dynamic Benamou-Brenier formulation, for key steps in the analysis. Working with GPT-5 took a few weeks of effort, and we estimate that it could have taken several months to get the same results otherwise.  At the same time, there are still areas where working with GPT-5 was challenging: it sometimes provided incorrect references, and glossed over details that sometimes took days of work to fill in.
 We outline our workflow and steps taken to mitigate issues.
 Overall, our work can serve as additional documentation for a new age of human-AI collaborative work in mathematical science.
}
%\clearpage
% \tableofcontents
% \medskip

%\input{intro}

\section{Introduction}

Over the last few months, AI models for text generation have improved greatly, with new releases in major families, including GPT 5, Gemini, Claude, etc.
Among a variety of use cases, 
there are now several documented examples where they have helped professional mathematical scientists prove new results, sometimes even helping resolve known open problems, see e.g., \cite{jin2025counting,feldman2025g,jang2025point,salim2025accelerating,bubeck2025early,alexeev2025asymptotically,alexeev2025forbidden}, etc.

In this short note, we add another example to the list, by documenting how we were able to solve a
previously unsolved research
problem in robust mathematical statistics---where publicly documented examples of AI-assisted proofs are relatively more sparse---with crucial help from GPT-5. 
Our problem concerns robust density estimation when the observations are perturbed by Wasserstein-bounded contaminations.
 In a previous preprint \citep[][v2]{chao2023statistical},
 we have obtained upper and lower bounds on the minimax optimal estimation error; which were, however, not sharp. 
 
 Starting in October 2025, making significant use of GPT-5 Pro, we were able to derive the minimax optimal error rate, showing that neither of the previous upper/lower bounds were tight.
 GPT-5 provided crucial help along the way, including  by suggesting techniques that were not familiar to us, such as the dynamic Benamou-Brenier formulation \citep{brenier2004extended}, for key steps in the analysis.
 
 Working with GPT-5 took a few weeks of effort, and we estimate that it could have taken several months to get the same results otherwise.  At the same time, there are still areas where working with GPT-5 was challenging: it provided incorrect references, and glossed over details that sometimes took days of work to fill in.
 We outline our workflow and steps taken to mitigate issues.
 Overall, our work can serve as additional documentation for a new age of human-AI collaborative work in mathematical science.

\section{The Problem: Robust Density Estimation under Wasserstein Contaminations}

In this section, we aim to describe the problem that we have considered at a level that should be accessible to those familiar with undergraduate multivariable calculus, probability theory, and statistics. 

Our problem\footnote{To the best of our knowledge, our question is indeed an unsolved problem. However, it is also not clear if anybody has formally proposed this precise problem before (beyond our own prior work on the problem). Thus, perhaps one might not call it an open problem. However, this reflects a dominant mode of thinking and research in statistics, where open problems have rarely been the main drivers of progress; see, for instance \cite{wasserman2012biggest, jordan2011openbayesian} for some informal discussion.} is the following: 

\begin{center}
\fbox{%
  \parbox{0.8\textwidth}{%
We observe an i.i.d.\ sample $X_1',\ldots,X_n'$ in $\R^p$ from a 
density $\tilde f$. However, this sample has been contaminated. 
This density $\tilde f$ of the data has been perturbed by a bounded contamination compared 
to the true density $f$. How well can we recover the true density 
based on the contaminated data?%
  }%
}
\end{center}

{\bf Wasserstein contaminations.}
The level of perturbation is assumed to be captured by a Wasserstein norm. 
Specifically, consider some $\ell_q$ norm, $q \in [1,\infty]$.
A typical case is the Euclidean norm, where $q=2$.
 Let $X_1,\ldots,X_n$ in $\R^p$ be an i.i.d.~sample from the 
true density $f$.
 Let $\ep\ge 0$ represent the strength of the perturbation.
 Then we assume that, for some $r\ge 1$, 
 there is a joint realization of the original and contaminated data, such that 
$(\E\|X_i'-X_i\|_q^r)^{1/r} \le \ep$, for all $i=1,\ldots, n$.
 Denoting by $W_{q,r}$ the Wasserstein-$r$ metric induced by the $\ell_q$ norm on $\R^p$, this corresponds to\footnote{The rigorous definition of the metric is on probability distributions instead of random variables. We use an informal definition here for simplicity.} $W_{q,r}(X_i',X_i) \le \ep$, see e.g., \cite{villani2003topics,villani2008optimal,santambrogio2015optimal,figalli2021invitation}, etc.

Now consider an arbitrary given $x_0 \in \R^p$. How well can we estimate the true density $f(x_0)$ based on the contaminated sample?
 And what is the best density estimator? 
 To quantify the level of error, we can choose a loss function,\footnote{Our result actually holds for much more general loss functions of the form $\varphi(|\hat f(x_0)-f(x_0)|)$, where $\varphi$ is a real-valued, non-negative, convex function that maps zero to zero and does not grow faster than an exponential, satisfying a doubling property, see \cite{chao2023statistical}. However, we will focus on the squared loss for clarity.} such as 
the squared error
 $L(\hat f, f) = (\hat f(x_0)-f(x_0))^2$.
\begin{figure}[t]
\centering
\begin{tikzpicture}[xscale=0.9,yscale=1.3]

  % Axes
  \draw[->] (-0.5,0) -- (6.8,0) node[right] {$x$};
  \draw[->] (0,-0.2) -- (0,2.4) node[above] {density};

  % True density f: wiggly (mixture of three bumps)
  \draw[thick,blue,domain=0.2:6.5,samples=300]
    plot (\x,{
      0.6*exp(-0.5*(\x-1.5)^2)
      +0.9*exp(-0.5*(\x-3.5)^2)
      +0.5*exp(-0.5*(\x-5.0)^2)
    });

  % Contaminated density \tilde f: also wiggly, slightly shifted/distorted
  \draw[thick,red,dashed,domain=0.2:6.5,samples=300]
    plot (\x,{
      0.5*exp(-0.5*(\x-1.4)^2)
      +1.0*exp(-0.5*(\x-3.7)^2)
      +0.4*exp(-0.5*(\x-5.1)^2)
    });
  \node[red] at (4.9,1.6) {$\tilde f(x)$};

  % Choose x_0 and mark f(x_0) on the true density
  \def\xzero{2.0}
  \def\yzero{0.83}  % approximate f(x_0) for the blue curve above

  % Vertical line up to f(x_0)
  \draw[densely dotted] (\xzero,0) -- (\xzero,\yzero);

  % Mark x_0 on x-axis
  \fill (\xzero,0) circle (1.4pt) node[below] {$x_0$};

  % Mark f(x_0) on the true density
  \fill (\xzero,\yzero) circle (1.4pt)
    node[above left] {$f(x_0)$};

  % Sample points from the contaminated density on the x-axis
  \foreach \x in {0.8,1.3,1.9,2.4,2.7,3.0,3.2,3.6,4.0,4.5,5.0}
    \fill (\x,0) circle (1.2pt);

  % Label for the contaminated sample
  \node[below] at (3.3,-0.35)
    {$X_1',\dots,X_n' \sim \tilde f$};

\end{tikzpicture}
\caption{True density $f$, contaminated density $\tilde f$, observed contaminated sample $X_1',\dots,X_n'$, and the target value $f(x_0)$ to be estimated.}
\end{figure}
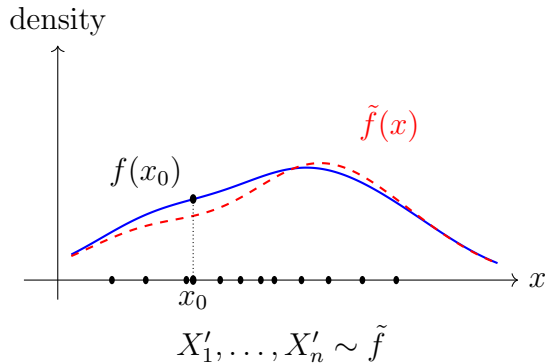

 This contamination model corresponds to small but systematic 
perturbations 
 of the entire original sample, and can be contrasted to the Huber $\epsilon$-contamination model, which allows large, potentially unbounded perturbations of a small fraction of the data, see e.g., \cite{huberLocation,huberRatio,huber2004robust,hampel2005robust}, etc. 
The study of the Wasserstein contamination model has been initiated in \cite{zhu2020robuststat,Liu2021RobustWE}; for problems other than density estimation.

\subsection{Background: Density Estimation and Minimax Rates of Convergence}

Without contaminations, 
the problem of density estimation has been widely studied in nonparametric statistics, see e.g., \cite{stone1984asymptotically,tsybakov2009nonparametric}.
 It is well known that some additional smoothness or structural 
 condition must hold in order for density estimation to be possible. 
 If the density is completely unconstrained, then it can wiggle or oscillate so quickly that there is no hope of estimating its value at any given point. 

{\bf Hölder smoothness.}
One of the most widely studied smoothness conditions is Hölder smoothness.
This requires that a function has sufficiently many derivatives,
which themselves have to be appropriately smooth. 
Consider some smoothness parameter $s>0$, and let $l$ be the 
largest non-negative integer that is strictly smaller than $s$.
A one-dimensional density is $s$-H\"older smooth if it has 
an $l$th derivative $f^{(l)}$, 
and there is $L>0$ (sometimes referred to as the Lipschitz constant)  such that
$$
|f^{(l)}(x)-f^{(l)}(y)|\le L|x-y|^{s-l} \ \text{for all }x,y\in\R.$$
For instance, if $s<1$, then $l=0$ and it is required that 
$$
|f(x)-f(y)|\le L|x-y|^{s} \ \text{for all }x,y\in\R.$$
This is interpreted as the  ``singularities" of $f$ being behaving at most as $x\mapsto x^s$, e.g., for $s = 1/2$ at most square-root singularities.

In the $p$-dimensional case, it is required that all partial derivatives up to order $l$ exist and satisfy the analogous condition with $\|x-y\|_2$ instead of $|x-y|$. 

{\bf Minimax risk.}
A fundamental way to evaluate performance in mathematical statistics and statistical decision theory is the minimax risk \citep{wald1949statistical}. 
This is the best performance that any method can achieve in the worst case over a specified problem class. 
In our setting, it is the best possible squared estimation error 
$$\E_{X_1', \ldots, X_n'\sim \tilde f} (\hat f(x_0; X_1', \ldots, X_n')-f(x_0))^2$$  
that an estimator $\hat f(x_0; X_1', \ldots, X_n')$
constructed 
using the contaminated sample $X_1', \ldots, X_n'$ can achieve, 
in the worst case over all $s$-H\"older
smooth densities with a given Lipschitz constant $L$
and all i.i.d $W_{q,r}$
perturbations $\tilde f$ of size at most $\varepsilon$ of the density $f$. 
To be fully precise, the perturbations are also required to be smooth for similar reasons as in the unperturbed case, however, they can have a potentially different Hölder smoothness constant and Lipschitz constant.\footnote{However, it turns out that the rate of convergence as a function of $n$ and $\epsilon$ does not depend on these, so they will not be mentioned further.} 

{\bf Classical rate of convergence.}
One of the cornerstones of nonparametric statistics 
is that without perturbations, the minimax optimal rate of convergence is \citep[see e.g.,][etc]{stone1984asymptotically,tsybakov2009nonparametric}
\begin{equation}\label{st}
n^{-\frac{2s}{2s+p}}.
\end{equation}

Moreover, this rate is achieved by kernel density estimation with 
a suitably smooth and bounded kernel and bandwidth \(
h^\star \asymp
  n^{-\frac{1}{2s+p}}.
  \)
\subsection{The New Result: Optimal Density Estimation under Wasserstein Contamination}

In this context, our new result on optimal density estimation under Wasserstein contaminations shows that the 
rate of the 
minimax risk
is determined by the maximum of two terms:
the standard rate and
a term dependent on $\varepsilon$.

\begin{theorem}[Density estimation under Wasserstein contamination, \cite{chao2023statistical}]\label{thm: pointwise density bounds}
Let $p\in\mathbb{N}$, $s>0$, $L>0$, 
$q\in[1,\infty]$,
$r\in[1,\infty)$. 
Then for all $n\ge1$ and $\ep\in(0,1]$,
the minimax risk for $s$-H\"older smooth density estimation in $p$ dimensions 
from
 $\ep$-bounded 
$W_{q,r}$ i.i.d.~Wasserstein contaminated data
has the rate 
\begin{equation*}
\max\left\{n^{-\frac{2s}{2s+p}},\ \ep^{\frac{2s}{ s+1+p/r }}\right\}.
\end{equation*}
\end{theorem}

{\bf Optimal estimator.}
Moreover, the upper bound is achieved by kernel density estimation with suitably smooth and bounded kernels and bandwidth \(
h^\star \asymp
 \max\left\{ n^{-\frac{1}{2s+p}},
  \ep^{\frac{1}{ s+1+p/r }}\right\}.
  \)
This shows that, compared to the unperturbed case,
the minimax risk and the required level of smoothing in the estimator does not change if $\ep$ is sufficiently small that
\(
\ep \lesssim n^{-\frac{s+1+p/r}{2s+p}}.
\)
However, the risk is dominated by the contaminations above that level, and the required level of smoothing is also determined by the magnitude of contaminations. 

{\bf Contrast with robustness in the classical Huber $\ep$-contamination model.}
To put things in context, one key reason why this result is interesting is that it shows that for the Wasserstein-type contaminations, classical methods based on local smoothing can be optimal in a specific problem. 
This is in stark contrast with the more classically studied Huber $\varepsilon$ contamination model, where 
standard smoothing or averaging-based estimators 
can be very sensitive, and other estimators, such as median-type methods, are required. 
Thus, complementing also the work of \cite{zhu2020robuststat,Liu2021RobustWE}, our results provide a different perspective on robust statistical estimation.

\section{The Proof: A Human-AI Interaction Workflow}

In this section, we explain our workflow for obtaining the proof.
 The full proof itself can be found in our preprint \citep[arxiv v3]{chao2023statistical}, and we will not repeat it here. 
 Instead, we will explain the context, our workflow emphasizes the complementary contributions of human work and AI, aiming to distill generalizable principles that could be useful more broadly. 

\subsection{Context and Background for the Problem}

{\bf Robust statistics.}
Our problem can be viewed as a natural question in robust statistics. 
Contaminations and robust statistics are well-studied, with a large literature focusing on the Huber $\ep$-contamination model, see e.g., \cite{huberLocation,huberRatio,huber2004robust,hampel2005robust}, etc. 
Given a distribution $P$, the data are sampled from a distribution $Q= (1-\ep)P+\ep P'$, where $P'$ is unconstrained in the standard formulation.

{\bf Optimal transport.}
Optimal transport 
has registered great progress over the last twenty years
\citep{villani2003topics,villani2008optimal,ambrosio2008gradient,santambrogio2015optimal,figalli2021invitation,peyre2019computational}.
Similarly, Wasserstein distributionally robust estimation and contamination have become widely studied, see e.g., \cite{esfahani2015dro,lee2017minimax,blanchet2017distributionally,singh2018minimax,manole2019minimax,weed2019estimation,blanchet2019robust,chen2020distributionally,staib2020learning,shafieezadeh2020wasserstein,hutter2021minimax,duchi2021statistics}.
These works typically focus on solving numerical minimax optimization problems algorithmically, where estimators minimize the least favorable risk over a class of contaminations.

{\bf Prior work on Wasserstein contaminations.}
In contrast, we aim to investigate robust estimation problems.
Prior work by 
\cite{zhu2020robuststat,Liu2021RobustWE} 
has initiated this line of inquiry, obtaining 
a variety of intriguing results, for problems ranging from mean estimation and linear regression to covariance matrix estimation.
Their works show that various estimators, such as Wasserstein GANs, can be nearly minimax optimal.

Inspired by these works,
starting around February 2021, we started investigating the problem of robustness to Wasserstein distribution shifts, with initial results reported in August 2023 \citep[arxiv v1]{chao2023statistical}.
While some of our results concern mean estimation and linear regression as studied in \cite{zhu2020robuststat,Liu2021RobustWE}, here we would like to focus on the problem of nonparametric estimation under Wasserstein distribution shift, which, to the best of our knowledge has not been studied before our work.

\subsection{Our Initial, Human-only Results}

Let $\M(\ep, s, p, q, r)$ be the minimax risk for $p$-dimensional $s$-H\"older smooth density estimation under $\ell_q^r$ Wasserstein perturbations of size $\ep$, as defined above. 
In the initial version of our manuscript \citep[arxiv v1]{chao2023statistical}, we have 
considered the special case of $p=1$-dimensional estimation, under $\ell_2^2$ perturbations, so that $q=r=2$.
We have
obtained the following result:
    \begin{align*}
        n^{-2s/(2s+1)} \vee \ep^{4s/(2s+1)}  \lesssim  \M(\ep, s, 1, 2, 2)  \lesssim n^{-2s/(2s+1)} \vee \ep^{2s/(s+2)} .
    \end{align*}
In this result, the term $n^{-2s/(2s+1)} $ represents the classical minimax optimal rate from \eqref{st}; and it is quite clear that this represents the correct dependence on the sample size $n$. 
However, the dependence on $\ep$ is much less clear.
Our results above do not match, and the lower bound can be much smaller in $\ep$ than the upper bound.

There are clear proof strategies for both the upper and the lower bound, which are well-known in the literature.
\begin{enumerate}
    \item For the {\bf upper bound}, a classical approach \cite[see e.g.,][etc]{stone1984asymptotically,tsybakov2009nonparametric} characterizes the bias and the variance of a kernel density estimator. 
    \item For the {\bf lower bound}, there are a variety of classical strategies, such as the Le Cam, Fano, and Assouad bounds, as well as the modulus of continuity
    \cite[see e.g.,][etc]{stone1984asymptotically,tsybakov2009nonparametric,donoho1991geometrizing}.
\end{enumerate}
In our initial work \citep[arxiv v1]{chao2023statistical}, we have started by attempting to adapt these strategies. Specifically:
\begin{enumerate}
    \item For the {\bf upper bound}, the same approach applies, but there is also a new contamination bias term.
    Let $K:\R^p \to \R$ be a kernel, and let $h>0$ be a bandwidth.
    Then this contamination bias term equals
    \begin{align}\label{cb}
        \frac{1}{h^p}\EE[X_1'\sim \tilde f]{K \left( \frac{X_1'-x_0} {h}\right)}-\frac{1}{h^p}\EE[X_1\sim f]{K \left( \frac{X_1-x_0} {h}\right)}. 
\end{align}
In our work, we have used a Lipschitz property for $K$ and Jensen's inequality to bound this term, obtaining a bias bound of the order of $\ep/h^2$.

\item  In our {\bf lower bound}, we have adapted the modulus of continuity approach, by starting with a normal density $f = \varphi_\sigma$ with sufficiently large variance $\sigma^2>0$, and perturbing it with a mean zero function to obtain a contaminated density $\tilde f$. 
Then, we upper bounded $W_{2,2}(f,\tilde f) \le \ep$ (via the Talagrand transportation cost inequality, see \cite{talagrand1996transportation,otto2000generalization}) and lower bounded $|f(x_0)-\tilde f(x_0)|$, which---by the modulus of continuity approach---lead to a lower bound on the estimation error.
    
\end{enumerate}

However, our two bounds did not match. 
Moreover, despite familiarity with basic techniques and results in optimal transport, we were unable to come up with better results. 
Further, we were even unsure whether the lower or the upper bounds need improvement, and simulation results were not too helpful, as the problem space is quite large (many possible base densities, perturbations, and estimators).
Therefore, we were quite stuck.

\subsection{Optimal Results via Human-AI Collaboration}

For the revision of the paper, 
starting in October 2025, the present author has decided to try using LLMs for assistance. 
LLM capabilities had been increasing steadily, and there were already a number of examples of LLMs being helpful with quite advanced mathematical results, as described above. 
This provided ample motivation.

{\bf Workflow.}
The workflow was simple. 
We copy-pasted the relevant parts of our LaTeX manuscript (problem statement and proof) into\footnote{It is possible that other models could work well; we have simply not tried these models in the current work.} GPT-5 via the web interface, and appended a simple prompt such as  ``I'm working on this theorem. Can you come up with a better upper bound? Please prepare ready to copy-paste fully rigorous LaTeX".
Then, we would send the prompt to GPT-5 Pro, copy-paste the output into a LaTeX editor, and carefully verify the results through our own human expertise.
If we identified any mistakes, we simply explained them to the model, and continued the conversation.

The extended test-time compute of GPT-5 Pro was essential, 
and reasoning took anywhere from a few minutes to 25 minutes per turn. Answers without test-time reasoning were significantly less useful.

{\bf The contributions of the AI.}
GPT helped by providing ideas and suggestions on improving both the upper and the lower bounds. 
In both cases, the argument started with the human-provided reasoning, and improved/sharpened the results.
We explain this below.

\begin{enumerate}

\begin{tcolorbox}[title=GPT Suggests Improved Bias Bound,fonttitle=\bfseries,colback=white,colframe=black]
User: [Prompt includes LaTeX draft of paper and a query: "Can you double-check this proof from beginning to end, and make sure it is fully correct? Please identify any place where there are incorrect arguments."]

GPT: [...]

The correct formulation is as follows. Define the interpolation
\begin{equation*}
T_t(x) = (1-t)x + t\,T(x), \qquad X_t = T_t(X),
\end{equation*}
where $X\sim P$ and $T$ is the optimal transport map from $P$ to $Q$. \red{Then
\begin{equation*}
\frac{d}{dt}\,\E\big[\varphi(X_t)\big]
=
\E\big[\varphi'(X_t)\cdot\big(T(X)-X\big)\big],
\end{equation*}
where the expectation is with respect to the original law of $X\sim P$}.
%Equivalently, in Eulerian form, the velocity field $v_t$ must satisfy
% $v_t\big(T_t(x)\big) = T(x)-x,
% $
%so that
%$v_t(z) = \E\big[T(X)-X \mid X_t = z\big]$ 
[...] 

\medskip

\noindent\textbf{Consequence.}
Once the derivative is written correctly, the displayed Cauchy–Schwarz step goes through. The correct bound is
\begin{equation*}
\big|\E_Q[\varphi] - \E_P[\varphi]\big|
\le
\Bigg(\int_0^1 \E_P\big[\varphi'(X_t)^2\big]\,dt\Bigg)^{1/2} W_2(P,Q),
\end{equation*}
which coincides with your statement after noting that
$
\E_P\big[g(X_t)\big] = \E_{P_t}[g]
$ for any test function $g$, where $P_t$ is the law of $X_t$.

\end{tcolorbox}

    \item  For the {\bf upper bound}, 
    {\bf the key suggestion was a better control of the contamination bias} term from \eqref{cb}.
    The first observation was to leverage that for our problem, and for\footnote{The remaining values of $q$ can be handled by simple limiting arguments, and $r=1$ can be handled by Kantorovich duality.} $q\in(1,\infty)$, $r>1$, there is always a deterministic optimal transport map $T$, as opposed to a generic probabilistic optimal coupling,  see e.g.,
\cite{GangboMcCann1996},
or
Theorem 10.28 or 10.38, and surrounding discussion in \citet{villani2008optimal}. We were aware of this result, but did not consider using it here. 

Write 
$\varphi(x)=K_h(x_0-x) = K[(x_0-x)/h]/h^p$, 
$\Delta(y)=T(y)-y$.
For $t\in[0,1]$, let $S_t(y)=(1-t)y+tT(y)$ being the linear interpolation of the identity map with the optimal transport map.
{\bf 
GPT suggested a sharper analysis of the contamination bias  $(K_h\!*\tilde f)(x_0)-(K_h\!*f)(x_0)$}
 by writing it leveraging the fundamental theorem of calculus through $  \varphi(T(y)) - \varphi(y)
    = \int_0^1 \nabla \varphi(S_t(y))^\top \Delta(y)\,dt$ as
\begin{equation*}
\int\big(\varphi(T(y))-\varphi(y)\big)f(y) dy
=\int_0^1\!\!\int \nabla\varphi\big(S_t(y)\big)^\top
 \Delta(y)  f(y) dy dt.
\end{equation*}
{\bf Consequence.}
Then, one can use H\"older's inequality to bound the inner product.
Let $\nu_t:=(S_t)_\#(f dy)$ 
be the pushforward measure 
and let its 
density be $\rho_t$.
Let $r^\star > 1$ be the conjugate of $r$, defined such that $1/r^\star+1/r = 1$, and let $q^\star$ be defined analogously.
For each $t\in(0,1)$, we then have
\begin{equation*}
\int \big|\nabla\varphi(S_t(y))^\top \Delta(y)\big|  f(y) dy
\ \le\
\Big(\int \|\nabla\varphi(z)\|_{q^\star}^{r^\star} \rho_t(z) dz\Big)^{1/r^\star}\Big(\int \|\Delta(y)\|_q^r f(y) dy\Big)^{1/r}.
\end{equation*}
If the density $\rho_t$ is bounded, then the first term can be bounded by a standard integral in terms of $\varphi$ against the Lebesgue measure, which depends only on the kernel $K$ (which is fixed), and the bandwidth $h$ (which can be controlled).
Moreover, the second term is precisely $W_{q,r}(\tilde f, f) \le \ep$  and so can be controlled.

This argument  leads to the correct upper bound.

{\bf Transcript and extract.}
We provide a link to the first transcript where GPT made such a suggestion in a correct form.\footnote{\url{https://chatgpt.com/share/6922876a-7988-8007-9c62-5f71772af6aa}} 
An extract from our interaction is also provided.\footnote{ 
Here, GPT actually suggests a fix to a previous, incorrect dynamic formulation.
Note that this is for a special case for our problem (one dimension, $q=r=2$), because it occurred prior to generalizing to a higher dimensional case, etc.}

\begin{tcolorbox}[title=GPT-5 Suggests Dynamic Benamou-Brenier for Improved Lower Bound (Extract) ,fonttitle=\bfseries,colback=white,colframe=black]

User: [Prompt includes LaTeX draft of paper and a query: "Can you go through this proof and see if you can improve either the upper bound and the lower bound until they match?"]

GPT: [...]

The root cause is that Lemma [..] scales like $W_2^2 \lesssim D^2\int g_{+}$, i.e. it is linear in the bump amplitude $\tau$ while the correct small–amplitude scaling for $W_2$ in one dimension is quadratic in $\tau$. 
% Using a sharp, amplitude–sensitive transportation bound improves the modulus by a factor $\tau^{1}$ and closes the gap.

[...]

\medskip

% \noindent\textbf{2. A refined local $W_2$ upper bound (sharp in amplitude)}

% \medskip

We state and prove the replacement for Lemma [..].

\begin{lemma}[Refined local $W_2$ bound via the dynamic formulation]
Let $q$ be a probability density on $\mathbb{R}$ and let $I\subset\mathbb{R}$ be a finite interval. Let $g\in L^1(\mathbb{R})$ satisfy $\operatorname{supp}(g)\subset I$ and $\int g=0$. Write $G(x)=\int_{-\infty}^x g(t)\,dt$. Suppose there exists $c_I>0$ such that $q(x)\ge c_I$ on $I$ and $|g|_\infty\le c_I/2$. Then $q\pm g$ are densities and
\[
W_2^2(q+g,q)\le 2c_I\int_{\mathbb{R}} G(x)^2\,dx,
\qquad
W_2^2(q-g,q)\le 2c_I\int_{\mathbb{R}} G(x)^2\,dx.
\]
\end{lemma}

\begin{proof}
%The assumptions ensure $q\pm g\ge 0$ and $\int(q\pm g)=1$ because $g$ is supported on $I$, $|g|_\infty\le c_I/2$, and $q\ge c_I$ on $I$. 
Consider the path of densities $\rho_t=q+t g$ for $t\in[0,1]$. On $I$ we have $\rho_t\ge q-|g|_\infty\ge c_I/2$, while off $I$ one has $\rho_t=q>0$. Define $v_t(x)=-G(x)/\rho_t(x)$. Then $\rho_t v_t=-G$ is independent of $t$, hence $\partial_x(\rho_t v_t)=-\partial_x G=-g$ and $\partial_t\rho_t=g$; thus $\partial_t\rho_t+\partial_x(\rho_t v_t)=0$. 
\red{Applying the Benamou–Brenier dynamic formulation gives
\[
W_2^2(q+g,q)
\le
\int_0^1\!\!\int_{\mathbb{R}} |v_t(x)|^2\,\rho_t(x)\,dx\,dt
=
\int_0^1\!\!\int_{\mathbb{R}} G(x)^2\,dx\,dt
=
\int_{\mathbb{R}} G(x)^2\,dx.
\]}
Since $\rho_t(x)\ge c_I/2$ on $I$ and $\rho_t=q$ off $I$, the bound sharpens to
\[
W_2^2(q+g,q)\le 2c_I\int_{\mathbb{R}} G(x)^2\,dx,
\]
and the same estimate holds for $W_2(q-g,q)$.
\end{proof}
[...]

\end{tcolorbox}

\item For the {\bf lower bound}, {\bf the key suggestion was to provide an improved bound on the magnitude of the perturbation $W_{q,r}(f,\tilde f)$, via the dynamic Benamou-Brenier formulation of optimal transport} \citep{brenier2004extended}.

Specifically, GPT made suggestions that we eventually refined into the following lemma:

\begin{lemma}[Local $W_{q,r}$ bound in $\R^p$ via the dynamic formulation]\label{lem:local-Wr-d2}
Let $\mathcal{G}$ be a probability density on $\R^p$, let $I=\prod_{j=1}^p[a_j,b_j]\subset\R^p$ be a bounded box, and let $g\in L^1(\R^p)$ satisfy $\mathrm{supp}(g)\subset I$ and $\int_{\R^p} g=0$. 
Assume $\mathcal{G}(x)\ge c_I>0$ on $I$ and $\|g\|_\infty\le c_I/2$. 
For $r\in[1,\infty)$, 
consider 
a vector field $U\in L^r(\R^p;\R^p)$ 
with $\mathrm{supp}(U)\subset I$ and $\nabla\!\cdot U=g$.
Then 
\[
W_{q,r}^r(\mathcal{G}\pm g,\mathcal{G})\ \le\ 
\frac{2^{ r-1}}{c_I^{ r-1}}\int_{\R^p}\|U(x)\|_q^r dx.
\]
\end{lemma}

To prove this statement GPT suggested 
using the dynamic Benamou-Brenier formulation.
This shows that the optimal transport cost is a lower bound on the transportation cost along any path that satisfies a continuity equation.

{\bf Suggested argument.}
GPT suggested the following argument to prove that this holds:
For $t\in[0,1]$ define $\rho_t:=\mathcal{G}+t g$ and $v_t(x):=-U(x)/\rho_t(x)$. 
Since $\rho_t(x)\ge c_I/2$ on $I$
and $U$ vanishes outside $I$, this is well defined.
We claim that $(\rho_t,v_t)_{t\in[0,1]}$ satisfies the continuity equation
\(
\partial_t\rho_t+\nabla\!\cdot(\rho_t v_t)=0\)
with endpoints $\rho_0=\mathcal{G}$ and $\rho_1=\mathcal{G}+g$.
Indeed $\partial_t\rho_t=g$, 
while
$\nabla\!\cdot(\rho_t v_t)=\nabla\!\cdot(-U)=-\nabla\!\cdot U=-g$,
hence the sum vanishes.

The results of \cite{brenier2004extended} then imply the lemma. 
Using the lemma in our theorem, and verifying the appropriate regularity conditions (also with help from GPT), lead to the desired lower bound.

{\bf Transcript and extract.}
We provide a link to the first transcript where GPT made such a suggestion.\footnote{\url{https://chatgpt.com/share/69227dd1-a00c-8007-8072-26b6e97a0820}} 
An extract from our interaction is also provided.\footnote{This is also for a special case for our problem (one dimension, $q=r=2$), because it occurred prior to generalizations. 
However, this was where the AI gave the initial idea.}

\item Beyond these concrete suggestions, we have also benefited from using GPT in other ways, which are now perhaps relatively well-established, and thus we will not dwell on them too much.
For instance, as it is now well-established in the literature \citep[see e.g.,][etc]{zheng2023judging},
{\bf we have asked the LLM to be a verifier by checking specific proofs} or writing, and to prepare a list of potential issues. 
These were a concern for our preprint because we went over several revisions and extensions, which could easily introduce inconsistencies. 

This was helpful for catching a number of small notational inconsistencies in our paper (e.g., using both $p$ and $d$ for the dimension).
The helpfulness stems from being able to human-verify the AI claims (clearly true e.g., for the notation issues) much more quickly than it would take to find them. 

\end{enumerate}

{\bf We think that the contributions of the AI were remarkable, and would certainly describe them as  ``creative" if they had been suggested by a human.}
After having worked on this problem for more than two years, and not having solved it, we can certainly appreciate the contribution.
The contribution of the AI was crucial, and we estimate that it has saved us up to several months of work.
In order to recover the same proofs, we would have essentially needed to read the literature, learn about all the tools, and realize how they can be used in our problem.

{\bf Isn't this all standard?}
It is likely that other researchers, with a deeper knowledge of optimal transport, will view these techniques as  ``standard".
In a sense, this note reveals the limitations of the author's knowledge and expertise more than anything.
However, we still believe that the contributions of AI are valuable, and that the setting that we described is somewhat typical of individual researchers' workflow. 

\begin{enumerate}
    \item 
First of all there are few researchers who have a deep knowledge of each of optimal transport, nonparametric density estimation, and robust statistics.
There are even fewer who have showed an active interest in the intersection of the three in terms of published research (we are not aware of any!), 
and who have a ready availability so that we could have asked for help or guidance. See Figure \ref{fig:venn-problem-space} for a cartoon illustration.

\item Second, in the mathematical sciences, it is quite typical to work in small teams, with up to four-five collaborators. 
Thus, necessarily, the expertise of most teams is limited. 
\end{enumerate}

In contrast, {\bf a key advantage of AI as a collaborator is that it has a vast knowledge, including of the three areas above, and that it is immediately available for consultation.}
The LLM has read a vast number of papers and books, and so can make connections between similar problems, results, and techniques, across a wide range. 
To put it another way, key advantages are its large memory, and its ready anytime availability.

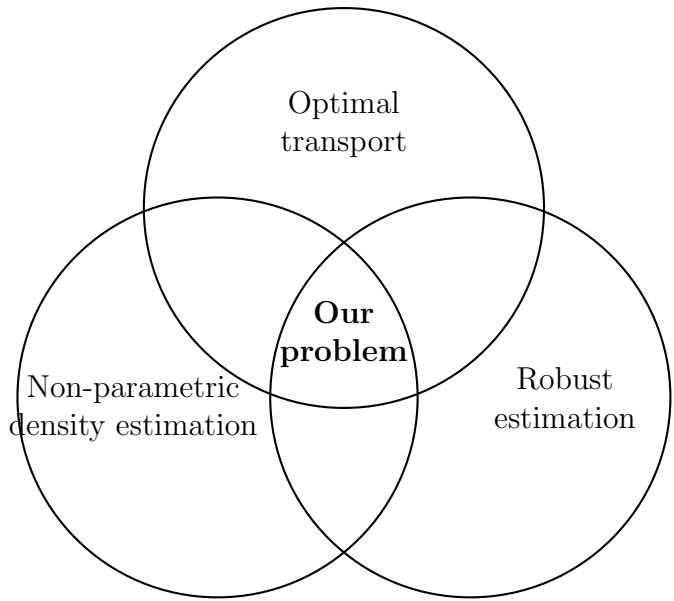
\begin{figure}[ht]
\centering
\begin{tikzpicture}[scale=1.4, every node/.style={align=center}]
  % circle radius
  \def\r{1.9}

  % centers of the three circles
  \coordinate (A) at (-1.2,0);
  \coordinate (B) at (1.2,0);
  \coordinate (C) at (0,1.8);

  % the three sets
  \draw[thick] (A) circle (\r);
  \draw[thick] (B) circle (\r);
  \draw[thick] (C) circle (\r);

  % labels for the three sets
  \node at (-2,-0.1) {Non-parametric\\ density estimation};
  \node at (2.1,0) {Robust\\ estimation};
  \node at (0,2.6) {Optimal\\ transport};

  % label for the triple intersection
  \node[font=\bfseries] at (0,0.6) {Our \\  problem};
\end{tikzpicture}
\caption{Schematic of the problem setting at the intersection of nonparametric density estimation, robust estimation, and optimal transport.}
\label{fig:venn-problem-space}
\end{figure}

\subsection{Limitations of the AI}

While the AI clearly made crucial contributions, it was definitely not perfect. 
It hallucinated references, it glossed over important details, and skipped key steps in the argument.
Below, we outline some of these issues, and how we eventually overcame them.

\begin{enumerate}
    \item For the {\bf upper bound}, there were two major issues.
    \begin{enumerate}
        \item A crucial component is 
        {\bf the existence of the optimal transport map $T$}. Initially, the AI suggested that this follows from certain results in the book by \cite{santambrogio2015optimal}.
    After carefully reading those results, we realized that they were only for the Euclidean case where $q=2$. 
    After re-prompting the AI pointing this out, it suggested Chapter 10 of the book by \cite{villani2008optimal}, which however, was quite hard to parse, due to tackling a far more general setting than we needed.
    
    Eventually, after additional AI-assisted search, we tracked down the reference by \cite{GangboMcCann1996}, a standard in the field. This shows that, {\bf even if a general suggestion is in the right direction, many details can be missing} and may need to be filled in by additional work. Ultimately, our human expertise was necessary to verify the correctness.
\item 
    Second, another crucial component of the argument is 
    {\bf the boundedness of the density $\rho_t$}. For this, the AI suggested another result from \cite{santambrogio2015optimal}, but that again only covered the Euclidean case. 

    Tracking down a correct reference proved tricky. Similar results are well-known in the literature, 
but most of them for slightly different assumptions; 
for instance 
Proposition 7.29 of \cite{santambrogio2015optimal} is for Euclidean (as opposed to $\ell_q$) norms,
 while the results of \cite{ohta2009finsler}
 are stated for squared (as opposed to power-$r$) distances. 

 {\bf The eventual solution.}
 After unsuccessfully querying the AI multiple times to generalize these results, and to find the required results in the literature, we were eventually successful by providing a sketch of the proof that the AI completed.
 Specifically, we realized that, 
 as in the 
 Riemannian Borell-Brascamp-Lieb 
 inequality
 \citep{cordero2001riemannian} and its adaptation to Finsler geometry \citep{ohta2009finsler},
  the determinant of the Jacobian of the transport map can be bounded along the interpolation path.
  In fact, given the properties of the optimal transport map \citep{GangboMcCann1996}, the proof was convenient.

  The proof of this part was highly iterative and interactive, requiring many steps of careful reading of the literature, prompting the AI, and verifying its solutions.
  There were many mis-steps, but they were all eventually resolved. This was the toughest part of the argument, where we thought that we hit a wall several times.

    \end{enumerate}
    
  \item For the {\bf lower bound}, the AI initially suggested to use the dynamic Benamou-Brenier formulation from the book by \cite{santambrogio2015optimal}.
  However, as we have already mentioned, that only covers the Euclidean case. 
  After the AI has not been able to find a correct reference even after multiple attempts, we were eventually led to check out the original paper by \cite{brenier2004extended}. Fortunately, their results were enough to cover our case.

  \item
  There were also a number of claims that we tried to prove that the AI struggled with.
  A priori, it was hard to tell whether or not it could work.
  {\bf Regardless of whether or not the AI found a correct proof of a claim, it would almost always claim that it had succeeded.} 
  This compelled us to verify its proof, which required time and effort. 
    Mistakes were often subtle, requiring careful attentive reading to find.
  
  {\bf When we found and reported a mistake, the AI would typically generate a longer and more complicated purported fix.} The cycle could repeat many times over. 
  Understandably, the AI is trained never to give up; which is often desirable, but sometimes giving up is better than running in circles forever.
  Our overall work required hundreds of queries.
  At some point, we had to use our best judgment to stop, give up, and exit. 
    
\end{enumerate}

Overall, we hope that the above discussion provides a balanced overview of the capabilities and limitations of using current publicly available AI systems for math, and what one may expect when using them.

\section{Commentary and Discussion}
We feel that several important comments must be made or repeated, and discussion points must be raised, in order to have a fully balanced report:
\begin{enumerate}
    \item {\bf Human expertise was indispensable, and the role of education was and remains crucial.} 
    For some background, over the last 15 years, the present author has acquired a good amount of experience with research in the mathematical sciences (reading and writing proofs, literature search), writing papers, and a good working knowledge of basic optimal transport, at the level of \cite{figalli2021invitation}. 
    These skills took  many years to develop, starting from the school level, through doctoral training, and continuing at the level of a faculty member.
    
    Without such skills, we feel that it would be effectively impossible to properly use current AI tools for advanced mathematical research. Therefore, {\bf current AI tools do not diminish in the slightest the importance of education}. Students should still be educated with the same basic mathematical skills.
    Even more, students with more advanced mathematical skills are likely to be able to leverage AI to much greater impact; hence the role of education becomes even more important.

    Instead, the implication on education is that, possibly, if one masters the basics of certain advanced fields or topics (e.g., the basics of existence theorems for optimal transport maps), one may be able to use AI as an aid to work at a more sophisticated level, tackling harder problems faster, than otherwise possible.
    Moreover, it is also possible that some level of training in the use of AI tools could help.

    \item {\bf The feeling of AI-empowerment.}
    Mathematical research requires
    coming up with creative ideas, 
    but can also involve significant amounts of intellectual effort on relatively routine but lengthy calculations (think of integrals requiring pages of computations); which are few mathematicians' favorite work.
Moreover, these calculations are often highly customized, and so the only help/automation that one could hope for is to ask a good collaborator or student to work on them.
    
    Since AI can excel at such calculations, it can greatly reduce  the amount of  required exhausting work. Of course, AI outputs need to be checked, but this is often less challenging than performing the calculations.
    This can be highly empowering. By saving energy from routine calculations, mathematicians can focus more on creative ideas, leading to an increase in both productivity and well-being.

\end{enumerate}

To summarize, in this note we have documented how an instance of human-AI collaboration led to the solution of a research problem in mathematical statistics. 
While current AI systems were not perfect, they already 
provided valuable suggestions and
greatly sped up the research process.

    {\small
        \setlength{\bibsep}{0.2pt plus 0.3ex}
        \bibliographystyle{plainnat-abbrev}
        \bibliography{references}

@article{jin2025counting,
  title={Counting Cycles with Deepseek},
  author={Jin, Jiashun and Ke, Tracy and Sui, Bingcheng and Wang, Zhenggang},
  journal={arXiv preprint arXiv:2505.17964},
  year={2025}
}

@article{alexeev2025forbidden,
  title={Forbidden Sidon subsets of perfect difference sets, featuring a human-assisted proof},
  author={Alexeev, Boris and Mixon, Dustin G},
  journal={arXiv preprint arXiv:2510.19804},
  year={2025}
}

@article{alexeev2025asymptotically,
  title={Asymptotically optimal approximate Hadamard matrices},
  author={Alexeev, Boris and Jasper, John and Mixon, Dustin G},
  journal={arXiv preprint arXiv:2511.14653},
  year={2025}
}

@article{zheng2023judging,
  title={Judging llm-as-a-judge with mt-bench and chatbot arena},
  author={Zheng, Lianmin and Chiang, Wei-Lin and Sheng, Ying and Zhuang, Siyuan and Wu, Zhanghao and Zhuang, Yonghao and Lin, Zi and Li, Zhuohan and Li, Dacheng and Xing, Eric and others},
  journal={Advances in neural information processing systems},
  volume={36},
  pages={46595--46623},
  year={2023}
}

@article{cordero2001riemannian,
  title={A Riemannian interpolation inequality {\`a} la Borell, Brascamp and Lieb},
  author={Cordero-Erausquin, Dario and McCann, Robert J and Schmuckenschl{\"a}ger, Michael},
  journal={Inventiones mathematicae},
  volume={146},
  number={2},
  pages={219--257},
  year={2001},
  publisher={Springer}
}

@article{jordan2011openbayesian,
  author  = {Jordan, Michael I.},
  title   = {A Message From the President: What Are the Open Problems in Bayesian Statistics?},
  journal = {ISBA Bulletin},
  volume  = {18},
  number  = {1},
  pages   = {1--4},
  year    = {2011}
}

@misc{wasserman2012biggest,
  author       = {Wasserman, Larry},
  title        = {The Biggest Unsolved Problem},
  year         = {2012},
  month        = {June},
  howpublished = {Blog post on *Normal Deviate*},
  note         = {Posted June 21, 2012},
  url          = {https://normaldeviate.wordpress.com/2012/06/21/90/}
}

@article{wald1949statistical,
  title={Statistical Decision Functions},
  author={Wald, Abraham},
  journal={The Annals of Mathematical Statistics},
  volume={20},
  number={2},
  pages={165--205},
  year={1949},
  publisher={Institute of Mathematical Statistics}
}

@article{stone1984asymptotically,
  title={An asymptotically optimal window selection rule for kernel density estimates},
  author={Stone, Charles J},
  journal={The Annals of Statistics},
  pages={1285--1297},
  year={1984},
  publisher={JSTOR}
}

@article{chao2023statistical,
  title={Statistical estimation under distribution shift: Wasserstein perturbations and minimax theory},
  author={Chao, Patrick and Dobriban, Edgar},
  journal={arXiv preprint arXiv:2308.01853},
  year={2023}
}

@article{feldman2025g,
  title={G$\backslash$" odel Test: Can Large Language Models Solve Easy Conjectures?},
  author={Feldman, Moran and Karbasi, Amin},
  journal={arXiv preprint arXiv:2509.18383},
  year={2025}
}

@misc{bubeck2025early,
      title={Early science acceleration experiments with GPT-5}, 
      author={Sébastien Bubeck and Christian Coester and Ronen Eldan and Timothy Gowers and Yin Tat Lee and Alexandru Lupsasca and Mehtaab Sawhney and Robert Scherrer and Mark Sellke and Brian K. Spears and Derya Unutmaz and Kevin Weil and Steven Yin and Nikita Zhivotovskiy},
      year={2025},
      eprint={2511.16072},
      archivePrefix={arXiv},
      primaryClass={cs.CL},
      url={https://arxiv.org/abs/2511.16072}, 
}

@article{salim2025accelerating,
  title={Accelerating mathematical research with language models: A case study of an interaction with GPT-5-Pro on a convex analysis problem},
  author={Salim, Adil},
  journal={arXiv preprint arXiv:2510.26647},
  year={2025}
}

@article{jang2025point,
  title={Point Convergence of Nesterov's Accelerated Gradient Method: An AI-Assisted Proof},
  author={Jang, Uijeong and Ryu, Ernest K},
  journal={arXiv preprint arXiv:2510.23513},
  year={2025}
}

@article{ohta2009finsler,
  title={Finsler interpolation inequalities},
  author={Ohta, Shin-ichi},
  journal={Calculus of Variations and Partial Differential Equations},
  volume={36},
  number={2},
  pages={211--249},
  year={2009},
  publisher={Springer}
}

@article{GangboMcCann1996,
  author  = {Wilfrid Gangbo and Robert J. McCann},
  title   = {The geometry of optimal transportation},
  journal = {Acta Mathematica},
  year    = {1996},
  volume  = {177},
  number  = {2},
  pages   = {113--161},
  doi     = {10.1007/BF02392620}
}

@incollection{brenier2004extended,
  title={Extended Monge-Kantorovich theory},
  author={Brenier, Yann},
  booktitle={Optimal Transportation and Applications: Lectures given at the CIME Summer School, held in Martina Franca, Italy, September 2-8, 2001},
  pages={91--121},
  year={2004},
  publisher={Springer}
}

@book{hampel2005robust,
  title={Robust statistics},
  author={Hampel, Frank R and Ronchetti, Elvezio M and Rousseeuw, Peter J and Stahel, Werner A},
  journal={Wiley series in probability and statistics},
  year={2005},
  publisher={Wiley}
}

@book{huber2004robust,
  title={Robust statistics},
  author={Huber, Peter J},
  year={2004},
  publisher={John Wiley \& Sons}
}

@book{ambrosio2008gradient,
  title={Gradient flows: in metric spaces and in the space of probability measures},
  author={Ambrosio, Luigi and Gigli, Nicola and Savar{\'e}, Giuseppe},
  year={2008},
  publisher={Springer Science \& Business Media}
}

@book{villani2008optimal,
  title={Optimal transport: old and new},
  author={Villani, C{\'e}dric},
  volume={338},
  year={2008},
  publisher={Springer Science \& Business Media}
}

@book{villani2003topics,
  title={Topics in optimal transportation},
  author={Villani, C{\'e}dric},
  year={2003},
  publisher={American Mathematical Soc.}
}

@article{talagrand1996transportation,
  title={Transportation cost for Gaussian and other product measures},
  author={Talagrand, Michel},
  journal={Geometric \& Functional Analysis GAFA},
  volume={6},
  number={3},
  pages={587--600},
  year={1996},
  publisher={Springer}
}

@article{chen2020distributionally,
  title={Distributionally robust learning},
  author={Chen, Ruidi and Paschalidis, Ioannis Ch and others},
  journal={Foundations and Trends{\textregistered} in Optimization},
  volume={4},
  number={1-2},
  pages={1--243},
  year={2020},
  publisher={Now Publishers, Inc.}
}

@book{figalli2021invitation,
  title={An invitation to optimal transport, Wasserstein distances, and gradient flows},
  author={Figalli, Alessio and Glaudo, Federico},
  publisher={EMS},
  year={2021}
}

@article{santambrogio2015optimal,
  title={Optimal transport for applied mathematicians},
  author={Santambrogio, Filippo},
  journal={Birk{\"a}user, NY},
  volume={55},
  number={58-63},
  pages={94},
  year={2015},
  publisher={Springer}
}

@article{donoho1991geometrizing,
  title={Geometrizing rates of convergence, III},
  author={Donoho, David L and Liu, Richard C},
  journal={The Annals of Statistics},
  pages={668--701},
  year={1991},
  publisher={JSTOR}
}

@book{tsybakov2009nonparametric,
  title={Introduction to Nonparametric Estimation},
  author={Tsybakov, Alexandre B},
  pages={1--76},
  year={2009},
  publisher={Springer}
}

@techreport{shafieezadeh2020wasserstein,
  title={Wasserstein Distributionally Robust Learning},
  author={Shafieezadeh Abadeh, Soroosh},
  year={2020},
  institution={EPFL}
}

@inproceedings{blanchet2017distributionally,
  title={Distributionally robust groupwise regularization estimator},
  author={Blanchet, Jose and Kang, Yang},
  booktitle={Asian Conference on Machine Learning},
  year={2017},
  organization={PMLR}
}

@article{duchi2021statistics,
author = {Duchi, John C. and Glynn, Peter W. and Namkoong, Hongseok},
title = {Statistics of robust optimization: A generalized empirical likelihood approach},
journal = {Mathematics of Operations Research},
volume = {46},
number = {3},
pages = {946-969},
year = {2021},
}

@article{blanchet2019robust,
  title={Robust Wasserstein profile inference and applications to machine learning},
  author={Blanchet, Jose and Kang, Yang and Murthy, Karthyek},
  journal={Journal of Applied Probability},
  volume={56},
  number={3},
  pages={830--857},
  year={2019},
  publisher={Cambridge University Press}
}

@article{peyre2019computational,
  title={Computational optimal transport: With applications to data science},
  author={Peyr{\'e}, Gabriel and Cuturi, Marco},
  journal={Foundations and Trends{\textregistered} in Machine Learning},
  volume={11},
  number={5-6},
  pages={355--607},
  year={2019},
  publisher={Now Publishers, Inc.}
}

@article{hutter2021minimax,
  title={Minimax estimation of smooth optimal transport maps},
  author={H{\"u}tter, Jan-Christian and Rigollet, Philippe},
  journal={The Annals of Statistics},
  volume={49},
  number={2},
  pages={1166--1194},
  year={2021},
  publisher={Institute of Mathematical Statistics}
}

@article{manole2019minimax,
author = {Tudor Manole and Sivaraman Balakrishnan and Larry Wasserman},
title = {{Minimax confidence intervals for the Sliced Wasserstein distance}},
volume = {16},
journal = {Electronic Journal of Statistics},
number = {1},
publisher = {Institute of Mathematical Statistics and Bernoulli Society},
pages = {2252 -- 2345},
year = {2022},
}

@phdthesis{staib2020learning,
  title={Learning and optimization in the face of data perturbations},
  author={Staib, Matthew James},
  year={2020},
  school={Massachusetts Institute of Technology}
}

@inproceedings{weed2019estimation,
  title={Estimation of smooth densities in Wasserstein distance},
  author={Weed, Jonathan and Berthet, Quentin},
  booktitle={Conference on Learning Theory},
  pages={3118--3119},
  year={2019},
  organization={PMLR}
}

@article{singh2018minimax,
  title={Minimax distribution estimation in Wasserstein distance},
  author={Singh, Shashank and P{\'o}czos, Barnab{\'a}s},
  journal={arXiv preprint arXiv:1802.08855},
  year={2018}
}

@article{otto2000generalization,
  title={Generalization of an inequality by Talagrand and links with the logarithmic Sobolev inequality},
  author={Otto, Felix and Villani, C{\'e}dric},
  journal={Journal of Functional Analysis},
  volume={173},
  number={2},
  pages={361--400},
  year={2000},
  publisher={Elsevier}
}

@inproceedings{lee2017minimax,
  title={Minimax Statistical Learning with Wasserstein distances},
  author={Jaeho Lee and Maxim Raginsky},
  booktitle={Neural Information Processing Systems},
  year={2017}
}

@article{Liu2021RobustWE,
    author = {Liu, Zheng and Loh, Po-Ling},
    title = "{Robust W-GAN-based estimation under Wasserstein contamination}",
    journal = {Information and Inference: A Journal of the IMA},
    volume = {12},
    number = {1},
    pages = {312-362},
    year = {2022},
}

@article{zhu2020robuststat,
author = {Banghua Zhu and Jiantao Jiao and Jacob Steinhardt},
title = {{Generalized resilience and robust statistics}},
volume = {50},
journal = {The Annals of Statistics},
number = {4},
publisher = {Institute of Mathematical Statistics},
pages = {2256 -- 2283},
year = {2022},
}

@article{huberLocation,
author = {Peter J. Huber},
title = {{Robust estimation of a location parameter}},
volume = {35},
journal = {The Annals of Mathematical Statistics},
number = {1},
publisher = {Institute of Mathematical Statistics},
pages = {73 -- 101},
year = {1964},
}

@article{huberRatio,
author = {Peter J. Huber},
title = {{A Robust Version of the Probability Ratio Test}},
volume = {36},
journal = {The Annals of Mathematical Statistics},
number = {6},
publisher = {Institute of Mathematical Statistics},
pages = {1753 -- 1758},
year = {1965},
}

@article{esfahani2015dro,
author = {Esfahani, Peyman and Kuhn, Daniel},
year = {2015},
month = {05},
pages = {},
title = {Data-Driven Distributionally Robust Optimization Using the Wasserstein Metric: Performance Guarantees and Tractable Reformulations},
volume = {171},
journal = {Mathematical Programming},
}
    }

\end{document}